\documentclass[11pt, reqno]{amsart}
\usepackage{amsthm,amssymb,amsmath,amsfonts}

\usepackage{lmodern}
\usepackage[T1]{fontenc}

\usepackage[foot]{amsaddr}

\usepackage{graphicx}

\newtheorem{theorem}{Theorem}
\newtheorem{lemma}{Lemma}
\newtheorem{proposition}{Proposition}
\newtheorem{corollary}{Corollary}
\newtheorem{claim}{\it Claim}

\theoremstyle{definition}

\newtheorem{definition}{\sc Definition}
\newtheorem*{notation}{\sc Notation}

\newtheorem*{remarks}{\sc Remarks}
\newtheorem*{remark*}{\sc Remark}
\newtheorem*{example*}{\sc Example}

\newcommand{\Real}{{\rm Re}\,}

\newcommand{\clos}{{\rm clos}}
\newcommand{\conv}{{\rm conv}}

\expandafter\def\expandafter\normalsize\expandafter{%
    \normalsize
    \setlength\abovedisplayshortskip{8pt}
    \setlength\belowdisplayshortskip{8pt}
}

\include{diagxy}

\begin{document}


\title[Strong Feller processes with measure-valued drifts]{Strong Feller processes with measure-valued drifts}

\date{}

\author{D.~Kinzebulatov}

\begin{abstract}

We construct a strong Feller process associated with $-\Delta + \sigma \cdot \nabla$,
with drift $\sigma$ in a wide class of measures (weakly form-bounded measures, e.g.~combining weak $L^d$ and Kato class  measure singularities), by exploiting a quantitative dependence of the smoothness of the domain  of an operator realization of $-\Delta + \sigma \cdot \nabla$ generating a holomorphic $C_0$-semigroup on $L^p(\mathbb R^d)$, $p>d-1$, on the value of the form-bound of $\sigma$.  
Our method admits extension to other types of perturbations of $-\Delta$ or $(-\Delta)^{\frac{\alpha}{2}}$, e.g.~to yield
new $L^p$-regularity results for Schr\"{o}dinger operators with form-bounded measure potentials.

\end{abstract}

\address{{\scriptsize Department of Mathematics, University of Toronto, 40 St.~George Str., Toronto, ON, M5S2E4, Canada}}

\email{damir.kinzebulatov@utoronto.ca}

\subjclass[2010]{35J15, 47D07 (primary), 35J75 (secondary)}

\keywords{Feller processes, measure-valued drift, regularity of solutions, non-local operators}

\maketitle

\subsection{}Let $\mathcal L^d$ be the Lebesgue measure on $\mathbb R^d$, $L^p=L^p(\mathbb R^d,\mathcal L^d)$, $L^{p,\infty}=L^{p,\infty}(\mathbb R^d,\mathcal L^d)$  
and $W^{1,p}=W^{1,p}(\mathbb R^d,\mathcal L^d)$ 
the standard Lebesgue, weak Lebesgue
and Sobolev 
spaces, $C^{0,\gamma}=C^{0,\gamma}(\mathbb R^d)$ the space of H\"{o}lder continuous functions ($0<\gamma<1$), $C_b=C_b(\mathbb R^d)$ the space of bounded continuous functions, endowed with the $\sup$-norm, 
$C_\infty \subset C_b$ the closed subspace of functions vanishing at infinity,
$\mathcal W^{s,p}$, $s>0$, the Bessel  space endowed with norm $\|u\|_{p,s}:=\|g\|_p$,  
$u=(1-\Delta)^{-\frac{s}{2}}g$, $g \in L^p$, $\mathcal W^{-s,p}$ the dual of $\mathcal W^{s,p}$, and  $\mathcal S=\mathcal S(\mathbb R^d)$ the L.~Schwartz space of test functions.
We denote by $\mathcal B(X,Y)$ the space of bounded linear operators between complex Banach spaces $X \rightarrow Y$, endowed with operator norm $\|\cdot\|_{X \rightarrow Y}$;  $\mathcal B(X):=\mathcal B(X,X)$. Set $\|\cdot\|_{p \rightarrow q}:=\|\cdot\|_{L^p \rightarrow L^q}$.
We denote by $\overset{w}{\rightarrow}$ the weak convergence of $\mathbb R^d$- or $\mathbb C^d$-valued measures on $\mathbb R^d$, and the weak convergence in a given Banach space.

By $\langle u,v\rangle$ we denote the inner product in $L^2$,
$$
\langle u,v\rangle = \langle u\bar{v}\rangle :=\int_{\mathbb R^d}u\bar{v}\mathcal L^d \qquad (u, v \in L^{2}).
$$


\subsection{}Let $d \geqslant 3$. The problem of constructing a Feller process having infinitesimal generator $-\Delta + b\cdot\nabla$, with singular drift $b:\mathbb R^d \rightarrow \mathbb R^d$,
has been thoroughly studied in the literature (cf.~\cite{AKR,KR} and references therein), motivated by
applications, as well as the search for the maximal (general) class of vector fields $b$ such that the associated process exists. This search culminated in the following classes of critical drifts:

\begin{definition}A vector field $b:\mathbb R^d \rightarrow \mathbb R^d$ is said to belong to $\mathbf{F}_{\delta}$, the class of form-bounded vector fields, if $b$ is  $\mathcal L^d$-measurable and
there exists $\lambda = \lambda_\delta > 0$ such that 
$$
\| b (\lambda - \Delta )^{-\frac{1}{2}} \|_{2 \rightarrow 2} \leqslant \sqrt{\delta}. 
$$
\end{definition}

\begin{definition}A vector field $b:\mathbb R^d \rightarrow \mathbb R^d$ is said to belong to the Kato class $\mathbf{K}^{d+1}_\delta$ if $b$ is  $\mathcal L^d$-measurable and there exists $\lambda = \lambda_\delta > 0$ such that
$$
\| b (\lambda - \Delta)^{-\frac{1}{2}} \|_{1 \rightarrow 1} \leqslant \delta.
$$
\end{definition}

\begin{figure}
\begin{equation*}
\bfig
\node a1(0,0)[\qquad L^p+L^\infty~(p>d)]
\node a2(200,500)[L^d+L^\infty]
\node a3(400,1000)[L^{d,\infty}+L^\infty]
\node a4(600,1500)[\mathbf{F}_{\delta^2}]
\node a5(1000,500)[\mathbf{F}_0]
\node b1(-600,1500)[\mathbf{K}_\delta^{d+1}]
\node b0(-400,1000)[\mathbf{K}_0^{{d+1}}]
\arrow[a1`a2;]
\arrow[a2`a3;]
\arrow[a3`a4;]
\arrow[a5`a4;]
\arrow[a2`a5;]
\arrow[a1`b0;]
\arrow[b0`b1;]
\efig
\end{equation*}
\footnotesize{
Here $\rightarrow$ stands for $\subsetneq$, inclusion of vector spaces. \\
The inclusions $L^d + L^\infty \subsetneq \mathbf{F}_0:=\bigcap_{\delta>0}\mathbf{F}_{\delta}$, 
$L^{d,\infty} + L^\infty \subsetneq  \bigcup_{\delta>0} \mathbf{F}_{\delta}$
follow from the Sobolev embedding theorem, and the Strichartz inequality with sharp constants \cite{KPS}, respectively.
}
\end{figure}

 
We have:

1) $b(x)=\sqrt{\delta} \frac{d-2}{2}x|x|^{-2} \in \mathbf{F}_{\delta}$ (Hardy inequality).

2) Also, if $|b(x)| \leqslant \mathbf{1}_{|x_1|<1}|x_1|^{s-1}$, where $0<s<1$, $x=(x_1,\dots,x_d)$, $\mathbf{1}_{|x_1|<1}$ is the characteristic function of $\{x:|x_1|<1\}$, 
then $b \in \mathbf{K}^{d+1}_0$. 
An example of a $b \in \mathbf{K}^{d+1}_\delta \setminus \mathbf{K}^{d+1}_0$ can be obtained e.g.~by modifying \cite[p.~250, Example 1]{AS}\footnote{The value of the relative bound $\delta$ plays a crucial role in the theory of $-\Delta + b\cdot\nabla$, e.g.~if $\delta>4$, then the uniqueness of solution of Cauchy problem for
$\partial_t-\Delta + \sqrt{\delta} \frac{d-2}{2}x|x|^{-2} \cdot \nabla$ fails in $L^p$, see \cite[Example 7]{KS}, see also comments below.}.
Examples 1), 2) demonstrate that
$\mathbf{K}_\delta^{d+1} \setminus \mathbf{F}_{\delta_1} \neq \varnothing$, and $\mathbf{F}_{\delta_1} \setminus \mathbf{K}_\delta^{d+1} \neq \varnothing.$

It is clear that
\begin{equation*}
b \in \mathbf{F}_{\delta}\,\, (\text{or }\mathbf{K}_{\delta}^{d+1}) \quad \Leftrightarrow \quad \varepsilon b \in \mathbf{F}_{\varepsilon\delta} \,\,(\text{respectively, }\mathbf{K}_{\varepsilon\delta}^{d+1}), \quad \varepsilon>0.
\end{equation*}
In particular, there exist $b \in \mathbf{F}_{\delta}$ ($\mathbf{K}_\delta^{d+1}$) such that $\varepsilon b \not\in \mathbf{F}_0$ ($\mathbf{K}_0^{d+1}$) for any $\varepsilon>0$ (cf.~examples above).
The vector fields in $\mathbf{F}_{\delta} \setminus \mathbf{F}_0$ and $\mathbf{K}^{d+1}_\delta \setminus \mathbf{K}^{d+1}_0$
have critical order singularities (i.e.~sensitive to multiplication by a constant), at isolated points or along hypersurfaces, respectively. 

\smallskip

Earlier, the Kato class $\mathbf{K}_\delta^{d+1}$, with $\delta>0$ sufficiently small (but nevertheless allowed to be positive), has been recognized as `the right one' for
the existence of the Gaussian upper and lower bounds on the fundamental solution of $-\Delta + b\cdot \nabla$, see \cite{S,Zh}; the Gaussian bounds yield an operator realization of  $-\Delta + b\cdot\nabla$ generating
a (contraction positivity preserving) $C_0$-semigroup in $C_\infty$ (moreover, in $C_b$), whose integral kernel is  the transition probability function of a Feller process.
In turn, $b \in \mathbf{F}_\delta$, $\delta<4$, 
ensures that $-\Delta + b\cdot\nabla$ is dissipative in $L^p$, $p>\frac{2}{2-\sqrt{\delta}}$ \cite{KS}; then, if $\delta<\min\{1,\bigl(\frac{2}{d-2}\bigr)^2\}$, the $L^p$-dissipativity allows to run a Moser-type iterative procedure of  \cite{KS}, which takes $p \rightarrow \infty$ and thus
produces
an operator realization of  $-\Delta + b\cdot\nabla$ generating
a $C_0$-semigroup in $C_\infty$, hence a Feller process.

The natural next step toward determining the general class of drifts $b$ `responsible' for the existence of an associated Feller process is to consider $b=b_1+b_2$, with $b_1 \in \mathbf{F}_{\delta_1}$, $b_2 \in \mathbf{K}^{d+1}_{\delta_2}$.
Although it is not clear how to reconcile the 
dissipativity in $L^p$ and the Gaussian bounds, it turns out that neither of these properties is responsible for the existence of the process; in fact, the process exists for any $b$ in the following class \cite{Ki}:


%

\begin{definition}A vector field $b:\mathbb R^d \rightarrow \mathbb R^d$ is said to belong to $\mathbf{F}_\delta^{\scriptscriptstyle \frac{1}{2}}$, the class of \textit{weakly} form-bounded vector fields, if $b$ is  $\mathcal L^d$-measurable, and there exists $\lambda = \lambda_\delta > 0$ such that
$$
\| |b|^\frac{1}{2} (\lambda - \Delta)^{-\frac{1}{4}} \|_{2 \rightarrow 2} \leqslant \sqrt{\delta}.
$$
\end{definition}

The class $\mathbf{F}_\delta^{\scriptscriptstyle \frac{1}{2}}$ has been introduced in \cite[Theorem 5.1]{S2}. We have
$$\mathbf{K}_\delta^{d+1} \subsetneq \mathbf{F}_\delta^{\scriptscriptstyle \frac{1}{2}}, \quad
\mathbf{F}_{\delta^2} \subsetneq \mathbf{F}_\delta^{\scriptscriptstyle \frac{1}{2}},$$ 
\begin{equation}
\label{sum_prop}
b \in  \mathbf{F}_{\delta_1}^{~} \text{ and } \mathsf{f} \in \mathbf{K}^{d+1}_{\delta_2}  \quad \Longrightarrow \quad b + \mathsf{f} \in \mathbf{F}^\frac{1}{2}_{\delta}, \; \sqrt{\delta} = \sqrt[4]{\delta_1} + \sqrt{\delta_2}
\end{equation}
(see \cite{Ki}).
In \cite{Ki}, the construction of the process  goes as follows:
the starting object is an operator-valued function ($b \in \mathbf{F}_\delta^{\scriptscriptstyle \frac{1}{2}}$)
\begin{align*}
&\Theta_p(\zeta,b):=(\zeta-\Delta)^{-1} \\&-(\zeta-\Delta)^{-\frac{1}{2}-\frac{1}{2q}}\underbrace{(\zeta-\Delta)^{-\frac{1}{2q'}}|b|^{\frac{1}{p'}}}_{\in \mathcal B(L^p)}
\underbrace{\bigl(1+b^{\frac{1}{p}}\cdot\nabla(\zeta-\Delta)^{-1}|b|^{\frac{1}{p'}}\bigr)^{-1}}_{\in \mathcal B(L^p)}  \underbrace{ b^{\frac{1}{p}}\cdot\nabla(\zeta-\Delta)^{-\frac{1}{2}-\frac{1}{2r}} }_{\in \mathcal B(L^p)}\, (\zeta-\Delta)^{-\frac{1}{2r'}},
\end{align*}
where $\Real \zeta> \frac{d}{d-1}\lambda_\delta$, $b^{\frac{1}{p}}:=b|b|^{\frac{1}{p}-1}$, $p$ is in a bounded open interval determined by the form-bound $\delta$ (and expanding to $(1,\infty)$ as $\delta \downarrow 0$), and $1<r<p<q$. Then (see \cite{Ki} for details)
$$\Theta_p(\zeta,b)=(\zeta+\Lambda_p(b))^{-1},$$
where $\Lambda_p(b)$ is an operator realization of $-\Delta + b\cdot\nabla$ generating a holomorphic $C_0$-semigroup $e^{-t\Lambda_p(b)}$ on $L^p$, and the very definition of $\Theta_p(\zeta,b)$
implies that 
the domain of $\Lambda_p(b)$ $$D(\Lambda_p(b)) \subset \mathcal W^{1+\frac{1}{q},p}, \quad \text{for any } q>p.$$  
The information about smoothness of $D(\Lambda_p(b))$ allows us to leap, by means of the Sobolev embedding theorem, from $L^p$, $p>d-1$, to $C_\infty$, while moving the burden of the proof of convergence in $C_\infty$ (in the Trotter's approximation theorem) to $L^p$, a space having much weaker topology (locally).
Then (see \cite{Ki}) $\Theta_p(\mu,b)|_{\mathcal S}=(\mu+\Lambda_{C_\infty}(b))^{-1}|_{\mathcal S}$, where $\Lambda_{C_\infty}(b)$ is an operator realization of $-\Delta + b\cdot\nabla$ generating a contraction positivity preserving $C_0$-semigroup on $C_\infty$, hence a Feller process.


\subsection{}The primary goal of this note is to extend the method in \cite{Ki} to weakly form-bounded measure drifts.

The study of measure perturbations of $-\Delta$ has a long history, see e.g.~
\cite{AM,SV}, where the $L^p$-regularity theory of $-\Delta$ (more generally, 
of a Dirichlet form) perturbed by a measure potential in the corresponding Kato class was developed, $1 \leqslant p <\infty$ 
(cf.~Corollary \ref{cor3} below). 

Recently, \cite{BC} constructed a strong Feller process associated with $-\Delta + \sigma\cdot \nabla$ with a $\mathbb R^d$-valued measure $\sigma$ in the Kato class $\bar{\mathbf{K}}_\delta^{d+1}$ (see definition below), for $\delta=0$, running  perturbation-theoretic techniques in $C_b$, thus
obtaining
e.g.~a Brownian motion drifting upward when penetrating certain fractal-like sets.
We strengthen their result in Theorem \ref{cor1} below.



%

\begin{definition}A $\mathbb C^d$-valued Borel measure $\sigma$ on $\mathbb R^d$ 
is said to belong to $\bar{\mathbf{F}}^{\scriptsize \frac{1}{2}}_\delta$, the class of weakly form-bounded measures, if 
 there exists $\lambda = \lambda_\delta > 0$ such that
$$
\int_{\mathbb R^d} \biggl((\lambda-\Delta)^{-\frac{1}{4}}(x,y) f(y)dy\biggr)^2 |\sigma|(dx) \leqslant \delta \|f\|^2_2,  \quad f \in \mathcal S.,
$$ 
where $|\sigma|:=|\sigma_1|+\dots+|\sigma_d|$ is the variation of $\sigma$.
Clearly,  $\mathbf{F}^{\scriptsize \frac{1}{2}}_\delta \subset \bar{\mathbf{F}}^{\scriptsize \frac{1}{2}}_\delta$.
\end{definition}

\begin{definition} 
A $\mathbb C^d$-valued Borel measure $\sigma$ on $\mathbb R^d$ 
is said to belong to the Kato class  $\bar{\mathbf{K}}_\delta^{d+1}$ if there exists $\lambda = \lambda_\delta > 0$ such that
$$
\sup_{x \in \mathbb R^d} \int_{\mathbb R^d}(\lambda - \Delta)^{-\frac{1}{2}}(x,y) |\sigma|(dy)
\leqslant \delta.
$$
\end{definition}

See \cite{BC} for examples of measures in $\bar{\mathbf{K}}_0^{d+1}$.

It is clear that $\mathbf{K}_\delta^{d+1} \subset \bar{\mathbf{K}}_\delta^{d+1}$. 
By Lemma \ref{approx_lem} below, $\bar{\mathbf{K}}_\delta^{d+1} \subset \bar{\mathbf{F}}^{\scriptsize \frac{1}{2}}_\delta$. 

\smallskip


\begin{figure*}
\begin{equation*}
\bfig


\node a1(0,0)[\qquad L^p+L^\infty~(p>d)]
\node a2(0,500)[L^d+L^\infty]
\node a3(0,1000)[L^{d,\infty}+L^\infty]
\node a4(0,1500)[\mathbf{F}_{\delta^2}]
\node a5(700,1000)[\mathbf{F}_0]
\node b1(-350,1500)[\mathbf{K}_\delta^{d+1}]
\node b0(-700,1000)[\mathbf{K}_0^{{d+1}}]
\node c1(0,2000)[\mathbf{F}_\delta^{\scriptsize \frac{1}{2}}]

\node d1(-700,1500)[\bar{\mathbf{K}}_0^{d+1}]
\node d2(-350,2000)[\bar{\mathbf{K}}_\delta^{d+1}]
\node d3(0,2500)[\mathbf{F}_{\delta_1}^{\scriptsize \frac{1}{2}}+ \bar{\mathbf{K}}_{\delta_2}^{d+1}]
\node d4(0,3000)[\bar{\mathbf{F}}_\delta^{\scriptsize \frac{1}{2}}]

\arrow[a1`a2;]
\arrow[a2`a3;]
\arrow[a3`a4;]
\arrow[a5`a4;]
\arrow[a2`a5;]
\arrow[a1`b0;]
\arrow[b0`b1;]
\arrow[a4`c1;]
\arrow[b1`c1;]

\arrow[b0`d1;]
\arrow[b1`d2;]
\arrow[c1`d3;]

\arrow[d1`d2;]
\arrow[d2`d3;]
\arrow[d3`d4;]


\efig
\end{equation*}

\footnotesize{The general classes of drifts studied in the literature in connection with operator $-\Delta + \sigma \cdot \nabla$. \\ Here we identify $b(x)$ with $b(x)\mathcal L^d$. }

\label{diag}

\end{figure*}

The operator-valued function $\Theta_p(\zeta,\sigma)$, $\Real \zeta> \frac{d}{d-1}\lambda_\delta$  (see above), `a candidate' for the resolvent of the desired operator realization of $-\Delta + \sigma \cdot \nabla$ generating a $C_0$-semigroup on $C_\infty$, is not well defined for a  $\sigma$ having  non-zero singular part. We modify the method in \cite{Ki}.
Also, in contrast to the setup of \cite{Ki}, a general $\sigma$ doesn't admit a monotone approximation by regular vector fields $v_k$ (i.e.~by $v_k\mathcal L^d$),  
which complicates the proof of convergence $\Theta_2(\zeta,v_k\mathcal L^d) \overset{s}{\rightarrow} \Theta_2(\zeta,\sigma)$ in $L^2$, needed to carry out the method. We overcome this difficulty using an important variant of the Kato-Ponce inequality by \cite{GO} (see also \cite{BL}) (Proposition \ref{lem5} below).

Our method depends on the fact that the operators $-\Delta$, $\nabla$ constituting $-\Delta + \sigma \cdot \nabla$ commute. 
In particular, our method admits a straightforward generalization to $(-\Delta)^{\frac{\alpha}{2}} + \sigma \cdot \nabla$, where $(-\Delta)^{\frac{\alpha}{2}}$ is the fractional Laplacian, $1<\alpha<2$, with measure $\sigma$ weakly form-bounded with respect to $\Delta^{\alpha-1}$, i.e.
$$
\int_{\mathbb R^d} \biggl((\lambda-\Delta)^{-\frac{\alpha-1}{4}}(x,y) f(y)dy\biggr)^2 |\sigma|(dx) \leqslant \delta \|f\|^2_2,  \quad f \in \mathcal S
$$
for some $\lambda = \lambda_\delta > 0$.
(We note that the potential theory of operator $-\Delta^{\frac{\alpha}{2}}$ perturbed by a drift in the corresponding Kato class, as well as its associated process, attracted a lot of attention recently, see \cite{BJ, CKSo, KSo} and references therein.)

In Theorems \ref{cor0}, \ref{cor1} (but not in Corollary \ref{cor3}) we assume that $\sigma$
admits an approximation by (weakly) form-bounded measures $\ll \mathcal L^d$ having the same form-bound $\delta$ (in fact, $\delta+\varepsilon$, for an arbitrarily small $\varepsilon>0$ independent of $k$). We verify this assumption for $\sigma = b\mathcal L^d + \nu$, 
$$b \mathcal L^d \in \bar{\mathbf{F}}^{\scriptscriptstyle \frac{1}{2}}_{\delta_1}, \qquad \nu \in \bar{\mathbf{K}}_{\delta_2}^{d+1}, \qquad \sqrt{\delta}:=\sqrt{\delta_1}+\sqrt{\delta_2},$$
but do not address, in this note, the issue of constructing such an approximation for a general $\sigma$; we also do not address the issue (we believe, related) of constructing weakly form-bounded vector fields whose singularities are principally different from those of $\mathbf{F}_{\delta_1^2} + \mathbf{K}_{\delta_2}^{d+1}$ (cf.~\eqref{sum_prop}).

%

%

%

\subsection{}We proceed to precise formulations of our results.

%
%
%
%
%
%
%
%


\begin{notation}
Let
\begin{equation}
\label{m_d}
m_{d}:=\inf_{\kappa>0} \sup_{\substack{x \neq y, \\ \Real\zeta>0}} \frac{|\nabla (\zeta-\Delta)^{-1}(x,y)|}{\bigl(\kappa^{-1}\Real\zeta-\Delta\bigr)^{-\frac{1}{2}}(x,y)}
\end{equation}
(note that $m_d$ is bounded from above by $\pi^{\frac{1}{2}} (2e)^{-\frac{1}{2}} d^\frac{d}{2} (d-1)^{\frac{1-d}{2}}<\infty$, see~\cite[(A.1)]{Ki}),
$$
 \mathcal J:=\biggl(1+\frac{1}{1+\sqrt{1-m_{d}\delta}},1+\frac{1}{1-\sqrt{1-m_{d}\delta}} \biggr).
$$
%
\end{notation}

\smallskip

\begin{theorem}[$L^p$-theory of $-\Delta + \sigma \cdot \nabla$]
\label{cor0} 

Let $d \geqslant 3$. Assume that $\sigma$ is a $\mathbb C^d$-valued Borel measure  in $\bar{\mathbf{F}}^{\scriptscriptstyle \frac{1}{2}}_{\delta}$ such that
$
\sigma = b \mathcal L^d + \nu,
$ where $b:\mathbb R^d \rightarrow \mathbb C^d$,  $$b \mathcal L^d \in \bar{\mathbf{F}}^{\scriptscriptstyle \frac{1}{2}}_{\delta_1}, \qquad \nu \in \bar{\mathbf{K}}_{\delta_2}^{d+1}, \qquad \sqrt{\delta}:=\sqrt{\delta_1}+\sqrt{\delta_2},$$ or, more generally {\rm(}see Lemma \ref{approx_lem} below{\rm)},
$\sigma \in \bar{\mathbf{F}}^{\scriptscriptstyle \frac{1}{2}}_{\delta}(\lambda)$
is such that
there exist 
$v_k \in C_0^\infty(\mathbb R^d, \mathbb C^d)$, $v_k\mathcal L^d \in  \bar{\mathbf{F}}^{\scriptscriptstyle \frac{1}{2}}_{\delta}(\lambda)$, 
$
v_k \mathcal L^d \overset{w}{\longrightarrow} \sigma.
$

\smallskip
If $m_{d}\delta<1$, then for every $p \in \mathcal J$:

\smallskip

{\rm(}\textit{i}{\rm)}  There exists a holomorphic $C_0$-semigroup $e^{-t \Lambda_p(\sigma)}$ in $L^p$ such that, possibly after replacing $v_k\mathcal L^d$'s with a sequence of their convex combinations (also weakly converging to measure $\sigma$),
 we have
$$e^{-t \Lambda_p (v_k \mathcal L^d)} \overset{s}{\rightarrow} e^{-t \Lambda_p(\sigma)} \text{ in $L^p$},$$
as $k \uparrow \infty$, where
$$\Lambda_{p}(v_k  \mathcal L^d):=-\Delta + v_k \cdot \nabla, \quad D(\Lambda_p(v_k \mathcal L^d))=W^{2,p}.$$

\smallskip

{\rm(}\textit{ii}{\rm)}~The resolvent set $\rho(-\Lambda_{p}(\sigma))$ contains a  half-plane
$
\mathcal O \subset \{\zeta \in \mathbb C: \Real \zeta>0\},
$
and the resolvent $(\zeta+\Lambda_p(\sigma))^{-1}$, $\zeta \in \mathcal O$, admits an extension by continuity to a bounded linear operator in $\mathcal B\left(\mathcal W^{-\frac{1}{r'},p},\mathcal W^{1+\frac{1}{q},p}\right),$
where $1 \leqslant r<\min\{2,p\}$,  $\max\{2,p\}<q$.

\smallskip

{\rm(}\textit{iii}{\rm)}~The domain of the generator
$
D\bigl(\Lambda_p(\sigma)\bigr)\subset \mathcal W^{1+\frac{1}{q},p}
$ 
for every $q>\max\{p,2\}$.

\end{theorem}

\begin{remarks}

\textbf{I.}~If $\sigma \ll \mathcal L^d$, then the interval $\mathcal J\ni p$ in Theorem \ref{cor0} can be extended, see \cite{Ki} (in \cite{Ki} we work directly in $L^p$, while in the proof of Theorem \ref{cor0} we 
have to first prove our convergence results in $L^2$, and then transfer them to $L^p$ (Proposition \ref{lem50_p}), hence the more restrictive assumptions on $p$).

%
%

\textbf{II}.~A straightforward modification of the proof of Theorem \ref{cor0} yields:

\begin{corollary}[$L^p$-theory of $-\Delta + \Psi$] 

\label{cor3}
Let $d \geqslant 3$. Assume that $\Psi$ is a $\mathbb C$-valued Borel measure 
such that
$$
\int_{\mathbb R^d} \biggl((\lambda-\Delta)^{-\frac{1}{2}}(x,y) f(y)dy\biggr)^2 |\Psi|(dx) \leqslant \delta \|f\|^2_2,  \quad f \in \mathcal S,
$$ 
for some $\lambda=\lambda_\delta>0$. We write  $\Psi \in \bar{\mathbf{F}}_{\delta}\bigl(\Delta,\lambda\bigr)$. Set $V_k:=\rho_k e^{\varepsilon_k \Delta} \Psi,$ $\varepsilon_k \downarrow 0,$ where $\rho_k \in C_0^\infty$, $0 \leqslant \rho_k \leqslant 1$, $\rho \equiv 1$ in $\{|x| \leq k\}$, $\rho \equiv 0$ in $\{|x| \geq k+1\}$, so that
$$V_k\mathcal L^d \in \bar{\mathbf{F}}_\delta(\Delta,\lambda) \text{ for all $k$},  \qquad V_k\mathcal L^d \overset{w}{\rightarrow} \Psi \text{ as $k \uparrow \infty$}$$ {\rm(}see Lemma \ref{approx_lem2} below{\rm)}.
If $\delta<1$, then for every $p \in \bigl(1+\frac{1}{1+\sqrt{1-\delta}},1+\frac{1}{1-\sqrt{1-\delta}} \bigr)$ there exists a holomorphic $C_0$-semigroup $e^{-t \Pi_p(\Psi)}$ in $L^p$ such that
$$e^{-t \Pi_p (V_k \mathcal L^d)} \overset{s}{\rightarrow} e^{-t \Pi_p(\Psi)} \text{ in $L^p$},$$
where $\Pi_{p}(V_k  \mathcal L^d):=-\Delta + V_k,$ $ D(\Pi_p(V_k \mathcal L^d))=W^{2,p},$
possibly after replacing $V_k\mathcal L^d$'s with a sequence of their convex combinations {\rm(}also weakly converging to $\Psi${\rm)}, and the domain of the generator $D\bigl(\Pi_p(\Psi)\bigr)\subset \mathcal W^{\frac{1}{q},p}, $
for any $q>\max\{2,p\}$.

%
%
%
%
%
%

\end{corollary}


Corollary \ref{cor3} extends the results in \cite{AM,SV} (applied to operator $-\Delta + \Psi$), where a real-valued $\Psi$ is assumed to be in the Kato class $\bar{\mathbf{K}}^d_\delta$ of measures (e.g.~delta-function concentrated on a hypersurface). One disadvantage of Corollary \ref{cor3}, compared to \cite{AM, SV}, is that it 
requires $|\Psi| \leqslant \delta (\lambda-\Delta)$ (in the sense of quadratic forms) rather than $\Psi_- \leqslant \delta(\lambda-\Delta+\Psi_+)$, where $\Psi=\Psi_+-\Psi_-$, $\Psi_{+}, \Psi_- \geqslant 0$.
%


%
%
%
%
%
%
%
%
%
%
%
%
%

\end{remarks}

The purpose of Theorem \ref{cor0} is to prove

\begin{theorem}[$C_\infty$-theory of $-\Delta + \sigma \cdot \nabla$]
\label{cor1}


Let $d \geqslant 3$. Assume that $\sigma$ is a $\mathbb R^d$-valued Borel measure in $\bar{\mathbf{F}}^{\scriptscriptstyle \frac{1}{2}}_{\delta}$ such that
$
\sigma = b \mathcal L^d + \nu,
$ where $b:\mathbb R^d \rightarrow \mathbb R^d$,  $$b \mathcal L^d \in \bar{\mathbf{F}}^{\scriptscriptstyle \frac{1}{2}}_{\delta_1}, \qquad \nu \in \bar{\mathbf{K}}_{\delta_2}^{d+1}, \qquad \sqrt{\delta}:=\sqrt{\delta_1}+\sqrt{\delta_2},$$
 or, more generally {\rm(}see Lemma \ref{approx_lem} below{\rm)},
$\sigma \in \bar{\mathbf{F}}^{\scriptscriptstyle \frac{1}{2}}_{\delta}(\lambda)$
is such that
there exist 
$v_k \in C_0^\infty(\mathbb R^d, \mathbb R^d)$, $v_k\mathcal L^d \in  \bar{\mathbf{F}}^{\scriptscriptstyle \frac{1}{2}}_{\delta}(\lambda)$, 
$
v_k \mathcal L^d \overset{w}{\longrightarrow} \sigma.
$

\smallskip
If $m_{d}\delta<\frac{2d-5}{(d-2)^2}$, then:

\smallskip  

{\rm (\textit{i})} There exists a positivity preserving contraction $C_0$-semigroup $e^{-t\Lambda_{C_\infty}(\sigma)}$ on $C_\infty$ such that 
, possibly after replacing $v_k\mathcal L^d$'s with a sequence of their convex combinations (also weakly converging to measure $\sigma$)
 we have
\begin{equation*}
e^{-t\Lambda_{C_\infty}(v_k\mathcal L^d)} \overset{s}{\longrightarrow} e^{-t\Lambda_{C_\infty}(\sigma)} \text{ in } C_\infty, \quad t \geqslant 0,
\end{equation*}
as $k \uparrow \infty$, where $$\Lambda_{C_\infty}(v_k\mathcal L^d):=-\Delta+v_k \cdot \nabla, \quad D(\Lambda_{C_\infty}(v_k\mathcal L^d))=C^2 \cap C_\infty.$$


{\rm(\textit{ii})}~{\rm[}Strong Feller property{\rm\,]}  $(\mu+\Lambda_{C_\infty}(\sigma))^{-1}|_{\mathcal S}$ can be extended by continuity to a bounded linear operator in $\mathcal B(L^p, C^{0,\gamma})$, $\gamma<1-\frac{d-1}{p}$, for every $d-1<p<1+\frac{1}{1-\sqrt{1-m_{d}\delta}}$. 

\smallskip

{\rm(\textit{iii})}~The
integral kernel $e^{-t\Lambda_{C_\infty}(\sigma)}(x,y)$ {\rm(}$x,y \in \mathbb R^d${\rm)} of $e^{-t\Lambda_{C_\infty}(\sigma)}$ determines the
(sub-Markov) transition probability function of a Feller process.

\end{theorem}

\begin{remark*}
If $\sigma \ll \mathcal L^d$, then the constraint on $\delta$ in Theorem \ref{cor1} can be relaxed, see \cite{Ki},
cf.~Remark I above.
\end{remark*}

\smallskip

\section{Approximating measures}
\label{approxsect}

\subsection{In Theorems \ref{cor0} and \ref{cor1}} Suppose 
$
\sigma = b \mathcal L^d + \nu,
$ where $b:\mathbb R^d \rightarrow \mathbb C^d$, $b \mathcal L^d \in \bar{\mathbf{F}}^{\scriptscriptstyle \frac{1}{2}}_{\delta_1}(\lambda)$, and $\nu \in \bar{\mathbf{K}}_{\delta_2}^{d+1}(\lambda).$
The following statement is a part of Theorems  \ref{cor0} and \ref{cor1}.

\begin{lemma}
\label{approx_lem}
There exist vector fields
$v_k \in C_0^\infty(\mathbb R^d, \mathbb C^d)$, $k=1,2,\dots$ such that

{\rm(1)} $v_k\mathcal L^d \in  \bar{\mathbf{F}}^{\scriptscriptstyle \frac{1}{2}}_{\delta}(\lambda)$, $\sqrt{\delta}:=\sqrt{\delta_1}+\sqrt{\delta_2}$, for every $k$, and

{\rm(2)} $
v_k \mathcal L^d \overset{w}{\longrightarrow} \sigma
$ as $k \uparrow \infty$.
\smallskip
\end{lemma}
\begin{proof}
We fix functions
$\rho_k \in C_0^\infty$, $0 \leqslant \rho_k \leqslant 1$, $\rho \equiv 1$ in $\{|x| \leq k\}$, $\rho \equiv 0$ in $\{|x| \geq k+1\}$, 
and define
$$
v_k \mathcal L^d:=b_k \mathcal L^d + \nu_k,
$$
where, for some fixed $\varepsilon_k \downarrow 0$,
$$
\nu_k:=\rho_k e^{\varepsilon_k \Delta} \nu, \quad b_k:=\rho_k e^{\varepsilon_k \Delta} b. 
$$
It is clear that 
$v_k \in C_0^\infty(\mathbb R^d,\mathbb R^d)$ and  $v_k \mathcal L^d  \overset{w}{\longrightarrow}  \sigma$ as $k \uparrow \infty$. 
Let us show that $\nu_k \in \bar{\mathbf{K}}^{d+1}_{\delta_2}(\lambda)$ for every $k$. 
Indeed, we have the following pointwise (a.e.) estimates on $\mathbb R^d$:
\begin{align*}
(\lambda-\Delta)^{-\frac{1}{2}} |\nu_k| \leqslant  (\lambda-\Delta)^{-\frac{1}{2}} |e^{\varepsilon_k\Delta}\nu| \leqslant (\lambda-\Delta)^{-\frac{1}{2}} e^{\varepsilon_k\Delta}|\nu| = 
e^{\varepsilon_k\Delta}(\lambda-\Delta)^{-\frac{1}{2}} |\nu|.
\end{align*}
Since
$
\|e^{\varepsilon_k\Delta}(\lambda-\Delta)^{-\frac{1}{2}} |\nu|\|_\infty \leqslant \|(\lambda-\Delta)^{-\frac{1}{2}} |\nu|\|_\infty
$ and, in turn,
$\|(\lambda-\Delta)^{-\frac{1}{2}} |\nu|\|_\infty \leqslant \delta_2$ ($\Leftrightarrow \nu \in \bar{\mathbf{K}}^{d+1}_{\delta_2}(\lambda)$),
we 
have $\nu_k \in \bar{\mathbf{K}}^{d+1}_{\delta_2}(\lambda)$. 
By interpolation, $\nu_k\in \bar{\mathbf{F}}^{\scriptscriptstyle \frac{1}{2}}_{\delta_1}(\lambda)$. 
A similar argument yields $b_k \mathcal L^d \in \bar{\mathbf{F}}^{\scriptscriptstyle \frac{1}{2}}_{\delta_1}(\lambda)$. Thus, $v_k\mathcal L^d \in  \bar{\mathbf{F}}^{\scriptscriptstyle \frac{1}{2}}_{\delta}(\lambda)$, for every $k$.
\end{proof}

\subsection{In Corollary \ref{cor3}}
Suppose
$
\Psi \in \bar{\mathbf{F}}_\delta(\Delta,\lambda).
$
Select $\rho_k \in C_0^\infty$, $0 \leqslant \rho_k \leqslant 1$, $\rho \equiv 1$ in $\{|x| \leq k\}$, $\rho \equiv 0$ in $\{|x| \geq k+1\}$. Fix some $\varepsilon_k \downarrow 0$.

\begin{lemma}
\label{approx_lem2}

We have $V_k:=\rho_k e^{\varepsilon_k \Delta} \Psi   \in C_0^\infty(\mathbb R^d)$, and

{\rm(1)} $V_k\mathcal L^d \in \bar{\mathbf{F}}_\delta(\Delta,\lambda)$ for every $k$, 

{\rm(2)} $V_k\mathcal L^d \overset{w}{\rightarrow} \Psi$ as $k \uparrow \infty$.

\end{lemma}

\begin{proof}
Assertion (2) is immediate. Let us prove (1).
It is clear that $V_k\mathcal L^d \in \bar{\mathbf{F}}_{\delta}\bigl(\Delta,\lambda\bigr)$
if and only if 
$$\langle |V_k|\varphi,\varphi \rangle \leqslant \delta \langle (\lambda-\Delta)^{\frac{1}{2}}\varphi,(\lambda-\Delta)^{\frac{1}{2}}\varphi \rangle, \qquad \varphi \in \mathcal S.$$
We have
$
|V_k|=\rho_k e^{\varepsilon_k\Delta}|\Psi| \leqslant e^{\varepsilon_k\Delta}|\Psi|, 
$
so
\begin{align*}
\langle|V_k|\varphi,\varphi \rangle & \leqslant \langle e^{\varepsilon_k\Delta}|\Psi| \varphi,\varphi \rangle
= \langle |\Psi|, e^{\varepsilon_k\Delta}(\varphi^2) \rangle \qquad \biggl(\text{since $\Psi \in \bar{\mathbf{F}}_\delta(\Delta)$}\biggr) \\
& \leqslant \delta \left\langle  \biggl((\lambda-\Delta)^{\frac{1}{2}}[e^{\varepsilon_k\Delta}(\varphi^2)]^{\frac{1}{2}} \biggr)^2 \right\rangle = \delta \left\langle  (\lambda-\Delta)[e^{\varepsilon_k\Delta}(\varphi^2)]^{\frac{1}{2}}, [e^{\varepsilon_k\Delta}(\varphi^2)]^{\frac{1}{2}} \right\rangle \\
& = \delta \langle e^{\varepsilon_k\Delta}\varphi^2\rangle + \delta \langle \nabla[e^{\varepsilon_k\Delta}(\varphi^2)]^{\frac{1}{2}}, \nabla[e^{\varepsilon_k\Delta}(\varphi^2)]^{\frac{1}{2}} \rangle \qquad \biggl (\text{we are using }\langle e^{\varepsilon_k\Delta}\varphi^2\rangle = \langle \varphi^2\rangle \biggr) \\
& =\delta \langle \varphi^2\rangle + \delta \langle (e^{\varepsilon_k\Delta}\varphi^2)^{-1} (e^{\varepsilon_k\Delta} \varphi\nabla \varphi)^2\rangle \qquad \biggl(\text{by H\"{o}lder inequality}\biggr) \\
& \leqslant  \delta \langle \varphi^2\rangle + \delta \langle e^{\varepsilon_k\Delta}(\nabla \varphi)^2\rangle = \langle (\lambda-\Delta)^{\frac{1}{2}}\varphi,(\lambda-\Delta)^{\frac{1}{2}}\varphi \rangle,
\end{align*}
as needed.
\end{proof}

\section{Proof of Theorem \ref{cor0}} 



\subsection*{Preliminaries}

\textbf{1.~}By Lemma \ref{approx_lem}, there exist vector fields
$v_k \in C_0^\infty(\mathbb R^d, \mathbb C^d)$, $k=1,2,\dots$, such that
$v_k\mathcal L^d \in  \bar{\mathbf{F}}^{\scriptscriptstyle \frac{1}{2}}_{\delta}(\lambda),$ $ \sqrt{\delta}:=\sqrt{\delta_1}+\sqrt{\delta_2},$
and
$
v_k \mathcal L^d \overset{w}{\longrightarrow} \sigma$ $\text{as } k \uparrow \infty.
$

\smallskip

\textbf{2.~}Due to the strict inequality $m_{d}\delta<1$, we may assume that the infimum $m_{d}$ (cf.~\eqref{m_d}) is attained, i.e.~there is $\kappa_{d}>0$
$$
|\nabla (\zeta-\Delta)^{-1}(x,y)| \leqslant m_{d} \biggl(\kappa_{d}^{-1}\Real\zeta-\Delta\biggr)^{-\frac{1}{2}}(x,y), \quad \text{ $x,y \in \mathbb R^d$, $ x \neq y$, $\Real\zeta>0$}.
$$

%

%
%
%
%
\textbf{3.~}Set
$
\mathcal O:=\{\zeta \in \mathbb C: \Real\,\zeta \geqslant \kappa_{d}\lambda_{\delta}\},
$

\subsection*{The method of proof}
We modify the method of \cite{Ki}.
Fix some $p \in \mathcal J$,
and some $r,q$ satisfying $ 1 \leqslant r <\min\{2,p\} \leqslant \max\{2,p\} < q.$
Our starting object is an operator-valued function
\begin{equation*}
\Theta_p(\zeta,\sigma):=(\zeta-\Delta)^{-\frac{1}{2}-\frac{1}{2q}} \Omega_p(\zeta,\sigma,q,r)(\zeta-\Delta)^{-\frac{1}{2r'}} \in \mathcal B(L^p), \quad \zeta \in \mathcal O,
\end{equation*}
which is `a candidate' for the resolvent of the desired operator realization $\Lambda_p(\sigma)$ of $-\Delta + \sigma \cdot \nabla$ on $L^p$.
Here
\begin{equation}
\label{omega_p_def}
\Omega_p(\zeta,\sigma,q,r):=\biggl(\Omega_2(\zeta,\sigma,q,r)\biggl|_{L^p \cap L^2}\biggr)_{L^p}^{\clos} \in \mathcal B(L^p),
\end{equation}
where, on $L^2$,
\begin{equation*}
\Omega_2(\zeta,\sigma,q,r):=(\zeta-\Delta)^{-\frac{1}{2}\bigl(\frac{1}{2}-\frac{1}{q} \bigr)}(1+Z_2(\zeta,\sigma))^{-1}(\zeta-\Delta)^{-\frac{1}{2}\bigl(\frac{1}{2}-\frac{1}{r'} \bigr)} \in \mathcal B(L^2),
\end{equation*}
\begin{align*}
Z_2(\zeta,\sigma) &h(x):=(\zeta-\Delta)^{-\frac{1}{4}} \sigma \cdot \nabla(\zeta-\Delta)^{-\frac{3}{4}} h(x) \\
=&
\int_{\mathbb R^d}  (\zeta-\Delta)^{-\frac{1}{4}}(x,y) \biggl( \int_{\mathbb R^d} \nabla (\zeta-\Delta)^{-\frac{3}{4}}(y,z) h(z) dz  \biggr) \cdot \sigma(y)dy, \quad x \in \mathbb R^d, \quad h \in \mathcal S,
\end{align*}
and $\|Z_2\|_{2 \rightarrow 2}  <1$, so $\Omega_2(\zeta,\sigma,q,r) \in \mathcal B(L^2)$, see Proposition \ref{prop2} below. We prove that $\Omega_p(\zeta,\sigma,q,r) \in \mathcal B(L^p)$ in Proposition \ref{lem50} below.


\medskip

We show that $\Theta_p(\zeta,\sigma)$ is  the resolvent of $\Lambda_p(\sigma)$ (assertion (\textit{i}) of Theorem \ref{cor0}) by verifying conditions of the Trotter approximation theorem:

1) $\Theta_{p}(\zeta,v_k \mathcal L^d)=(\zeta+\Lambda_p(v_k\mathcal L^d))^{-1}$, $\zeta \in \mathcal O$, where $\Lambda_{p}(v_k \mathcal L^d):=-\Delta + v_k \cdot \nabla$, $D(\Lambda_p(v_k \mathcal L^d))=W^{2,p}.$

2) $\sup_{n \geqslant 1}\|\Theta_{p}(\zeta,v_k \mathcal L^d)\|_{p \rightarrow p} \leqslant C_p|\zeta|^{-1}$, $\zeta \in \mathcal O$. 

3) 
$\mu \Theta_{p}(\zeta,v_k \mathcal L^d) \overset{s}{\rightarrow} 1 \text{ in $L^p$ as } \mu\uparrow \infty \text{ uniformly in $k$}.$

4) $\Theta_{p} (\zeta,v_k \mathcal L^d) \overset{s}{\rightarrow} \Theta_{p} (\zeta,\sigma)$ in $L^p$ for some $\zeta \in \mathcal O$ as $k \uparrow \infty$ (possibly after replacing $v_k\mathcal L^d$'s with a sequence of their convex combinations, also weakly converging to measure $\sigma$), see Propositions \ref{prop_first} - \ref{lem50_p} below for details.

We note that a priori in 1) the set of $\zeta$'s for which $\Theta_{p}(\zeta,v_k \mathcal L^d)=(\zeta+\Lambda_p(v_k\mathcal L^d))^{-1}$ may depend on $k$; the fact that it actually does not is the content of Proposition \ref{last_prop}. 

The proofs of 2), 3), contained in Proposition \ref{prop_first} and \ref{conv_prop}, are based on an explicit representation of $\Omega_p(\zeta,v_k\mathcal L^d,q,r)$, $k=1,2,\dots$, see formula \eqref{alt_repr} below. (The representation \eqref{alt_repr}
doesn't exist if $\sigma$ has a non-zero singular part; we have to take a detour via $L^2$,
(cf.~\eqref{omega_p_def}), 
 which requires us to put somewhat more restrictive assumptions on $\delta$ (compared to \cite{Ki}, where the case of a $\sigma$ 
 having zero singular part is treated).)

Next, 4) follows from $\Theta_{2} (\zeta,v_k \mathcal L^d) \overset{s}{\rightarrow} \Theta_{2} (\zeta,\sigma)$, combined with $\sup_{n} \|\Theta_{p}(\zeta,v_k \mathcal L^d)\|_{2(p-1) \rightarrow 2(p-1)}<\infty$ ($\Leftarrow 2)$) and H\"{o}lder inequality, see Proposition \ref{lem50_p}. Our proof of $\Theta_{2} (\zeta,v_k \mathcal L^d) \overset{s}{\rightarrow} \Theta_{2} (\zeta,\sigma)$ (Proposition \ref{lem5}) uses the Kato-Ponce inequality by \cite{GO}.

\smallskip

Finally, we note that the very definition of the operator-valued function $\Theta_p(\zeta,\sigma)$ ensures smoothing properties $\Theta_p(\zeta,\sigma) \in \mathcal B\biggl(\mathcal W^{-\frac{1}{r'},p},\mathcal W^{1+\frac{1}{q},p}\biggr)$ $\Rightarrow$ assertion (\textit{ii}). Assertion (\textit{iii}) is immediate from (\textit{ii}).


%
%
%
%

\bigskip

Now, we proceed to formulating and proving Propositions \ref{prop2} - \ref{lem50_p}.

\begin{proposition}
\label{prop2} We have for every $\zeta \in \mathcal O$ \

\smallskip

{\rm(1)} $\|Z_2(\zeta,v_k \mathcal L^d)\|_{2 \rightarrow 2} \leqslant \delta$ for all $k$.

\smallskip

{\rm(2)} $\|Z_2(\zeta,\sigma)f\|_{2} \leqslant \delta\|f\|_2$, for all $f \in \mathcal S$, all $k$.

\end{proposition}
\begin{proof}
(1) Define $H:=|v_k|^{\frac{1}{2}}(\zeta-\Delta)^{-\frac{1}{4}}$, $ S:=v_k^{\frac{1}{2}}\nabla(\zeta-\Delta)^{-\frac{3}{4}}$ where $v_k^{\frac{1}{2}}:=|v_k|^{-\frac{1}{2}}v_k$. Then $
Z_2(\zeta,v_k \mathcal L^d) = H^{\ast}S,$
and we have
$$
\|Z_2(\zeta,v_k \mathcal L^d)\|_{2 \rightarrow 2} \leqslant \|H\|_{2 \rightarrow 2} \|S\|_{2 \rightarrow 2} \leqslant \|H\|_{2 \rightarrow 2} ^2 \|\nabla (\zeta-\Delta)^{-\frac{1}{2}}\|_{2 \rightarrow 2} \leqslant \delta,
$$
where $\|\nabla (\zeta-\Delta)^{-\frac{1}{2}}\|_{2 \rightarrow 2}=1$, and $\|H\|_{2 \rightarrow 2 }^2 \leqslant \delta$ (cf.~Lemma \ref{approx_lem}(1)).
\smallskip

(2) We have, for every $f$, $g \in \mathcal S$,
\begin{align*}
\bigl\langle g, Z_2(\zeta,\sigma)f \bigr\rangle = & \bigl\langle (\zeta-\Delta)^{-\frac{1}{4}}g, \sigma \cdot \nabla (\zeta-\Delta)^{-\frac{3}{4}}f \bigr\rangle \\
& \text{(here we are using $v_k \mathcal L^d \overset{w}{\rightarrow} \sigma$)} \\
& =
\lim_k \bigl\langle (\zeta-\Delta)^{-\frac{1}{4}}g, v_k \cdot \nabla (\zeta-\Delta)^{-\frac{3}{4}}f \bigr\rangle \\
& \text{(here we are using assertion (1))} \\
& \leqslant \delta \|g\|_2 \|f\|_2,
\end{align*}
i.e.~$\|Z_2(\zeta,\sigma)f\|_2 \leqslant \delta \|f\|_2$, as needed.
\end{proof}


The natural extension of $Z_2(\zeta,\sigma)|_{\mathcal S}$ (by continuity) to $\mathcal B(L^2)$ will be denoted again by $Z_2(\zeta,\sigma)$.
Since $\|Z_2(\zeta,v_k \mathcal L^d)\|_{2 \rightarrow 2}, \|Z_2(\zeta,\sigma)\|_{2 \rightarrow 2} \leqslant \delta<1$, we have $ \Omega_2(\zeta,v_k\mathcal L^d,q,r),  \Omega_2(\zeta,\sigma,q,r) \in \mathcal B(L^2)$. 

\medskip

%
%

Set 
$$\mathcal I:=\left(\frac{2}{1 + \sqrt{1-m_{d} \delta}}, \frac{2}{1 - \sqrt{1-m_{d} \delta}}\right).$$
In the next few propositions, given a $p \in \mathcal I$, we assume
$r,q$ satisfy $ 1 \leqslant r <\min\{2,p\} \leqslant \max\{2,p\} < q.$


\smallskip

The following proposition plays a principal role:

\begin{proposition}
\label{prop_first}
Let $p \in \mathcal I$. There exist constants $C_{p}$, $C_{p,q,r}<\infty$ such that for every $\zeta \in \mathcal O$

\smallskip
{\rm(1) }$
\|\Omega_p(\zeta,v_k\mathcal L^d,q,r)\|_{p \rightarrow p} \leqslant C_{p,q,r}$ for all $k$,

\smallskip

{\rm(2)} $ \|\Omega_p(\zeta,v_k\mathcal L^d,\infty,1)\|_{p \rightarrow p} \leqslant C_{p}|\zeta|^{-\frac{1}{2}}
$ for all $k$.

\end{proposition}

\begin{proof}

Denote $v_k^{\frac{1}{p}}:=|v_k|^{\frac{1}{p}-1}v_k$.
Set:
\begin{equation}
\label{alt_repr}
\tilde{\Omega}_p(\zeta,v\mathcal L^d,q,r):=Q_p(q)(1+T_p)^{-1} G_p(r), \quad \zeta \in \mathcal O,
\end{equation}
where
$$
Q_p(q):=(\zeta-\Delta)^{-\frac{1}{2q'}}|v_k|^{\frac{1}{p'}}, \quad T_p:=v_k^{\frac{1}{p}} \cdot \nabla (\zeta-\Delta)^{-1}|v|^{\frac{1}{p'}}, \quad G_p(r):=v_k^{\frac{1}{p}}\cdot \nabla (\zeta-\Delta)^{-\frac{1}{2}-\frac{1}{2r}},
$$
are uniformly (in $k$) bounded in $\mathcal B(L^p)$, and, in particular, $\|T_p\|_{p \rightarrow p} \leqslant 
\frac{pp'}{4} m_{d} \delta$ (see the proof of \cite[Prop.~1(\textit{i})]{Ki}), and $\frac{pp'}{4} m_{d} \delta<1$ since $p \in \mathcal I$. It follows that $C_{p,q,r}:=\sup_k \|\tilde{\Omega}_p(\zeta,v\mathcal L^d,q,r)\|_{p \rightarrow p}<\infty$. 
Now, $\tilde{\Omega}_p|_{L^2 \cap L^p} = \Omega_2|_{L^2 \cap L^p}$ (by expanding  $(1+T_p)^{-1}$, $(1+Z_2)^{-1}$ in the K.~Neumann series in $L^p$ and in $L^2$, respectively). Therefore,
$\tilde{\Omega}_p = \Omega_p$ $\Rightarrow$ assertion (1).
The proof of assertion (2) follows closely the proof of \cite[Prop.~1(\textit{ii})]{Ki}.
%
%
%
\end{proof}
%


Clearly,
$
\Theta_p(\zeta,v_k \mathcal L^d)
$ 
does not depend on $q$, $r$. Taking $q=\infty$, $r=1$,
we obtain from Proposition \ref{prop_first}:
\begin{equation}
\label{omega_est}
\|\Theta_p(\zeta,v_k \mathcal L^d)\|_{p \rightarrow p} \leqslant C_p|\zeta|^{-1}, \quad \zeta \in \mathcal O.
\end{equation}


%
%

\begin{proposition}
\label{last_prop}

Let $p \in \mathcal I$.
For every $k=1,2,\dots$ $\mathcal O \subset \rho(-\Lambda_p(v_k\mathcal L^d)),$ the resolvent set of $-\Lambda_p(v_k\mathcal L^d)$, and
$$\Theta_{p}(\zeta,v_k \mathcal L^d)=(\zeta+\Lambda_p(v_k\mathcal L^d))^{-1}, \quad \zeta \in \mathcal O,$$
where 
 $\Lambda_{p}(v_k \mathcal L^d):=-\Delta+v_k \cdot \nabla$, $ D(\Lambda_{C_\infty}(v_k \mathcal L^d))=W^{2,p}.$

%
\end{proposition}

 \begin{proof}
The proof repeats the proof of {\cite[Prop.~4]{Ki}}.
 \end{proof}

\begin{proposition}
\label{conv_prop} 
For $p \in \mathcal I$,
$\mu \Theta_{p}(\mu,v_k\mathcal L^d) \overset{s}{\rightarrow} 1 \text{ in $L^p$ as } \mu\uparrow \infty \text{ uniformly in $k$}.$

\end{proposition}
\begin{proof}
The proof repeats the proof of {\cite[Prop.~3]{Ki}}.
\end{proof}

\begin{proposition}
\label{lem5}
There exists a sequence $\{\hat{v}_n\} \subset \conv\{v_k\} \subset C_0^\infty(\mathbb R^d,\mathbb R^d)$ such that
\begin{equation}
\label{conv_80}
\hat{v}_n \mathcal L^d  \overset{w}{\longrightarrow}  \sigma \text{ as } n \uparrow \infty,
\end{equation}
and 
\begin{equation}
\label{conv_5}
\Omega_2(\zeta,\hat{v}_n \mathcal L^d,q,r)  \overset{s}{\rightarrow}  \Omega_2(\zeta,\sigma,q,r) \text{ in }L^2, \quad \zeta \in \mathcal O.
\end{equation}
\end{proposition}
\begin{proof}
To prove \eqref{conv_5}, it suffices to establish convergence
$
Z_2(\zeta,\hat{v}_n \mathcal L^d ) \overset{s}{\rightarrow } Z_2(\zeta,\sigma) \text{ in } L^2,$  $\zeta \in \mathcal O.
$

Let $\eta_r \in C_0^\infty$, $0 \leqslant \eta_r \leqslant 1$, $\eta_r \equiv 1$ on $\{x \in \mathbb R^d:|x| \leqslant r\}$ and $\eta_r \equiv 0$ on $\{x \in \mathbb R^d:|x| \geqslant r+1\}$.

\begin{claim}
\label{jclaim}
We have

\smallskip

{\rm(}\textit{j}{\rm)} $\|(\zeta-\Delta)^{-\frac{1}{4}} |v_k| (\zeta-\Delta)^{-\frac{1}{4}}\|_{2 \rightarrow 2} \leqslant \delta$ for all $k$.

\smallskip

 {\rm(\textit{jj})}  $\|(\zeta-\Delta)^{-\frac{1}{4}} |\sigma| (\zeta-\Delta)^{-\frac{1}{4}}f\|_{2 } \leqslant \delta\|f\|_2$, for all $f \in \mathcal S$.
\end{claim}
\begin{proof}
Define $H:=|v_k|^{\frac{1}{2}}(\zeta-\Delta)^{-\frac{1}{4}}$.
We have
$\|(\zeta-\Delta)^{-\frac{1}{4}} |v_k| (\zeta-\Delta)^{-\frac{1}{4}}\|_{2 \rightarrow 2} = \|H^*H\|_{2 \rightarrow 2} = \|H\|_{2 \rightarrow 2}^2 \leqslant \delta,$
where $\|H\|_{2 \rightarrow 2}^2 \leqslant \delta \,(\Leftrightarrow \, v_k \mathcal L^d \in \bar{\mathbf{F}}_\delta^{\scriptsize \frac{1}{2}}(\lambda)$, cf.~Lemma \ref{approx_lem}(1)), i.e.~we have proved (\textit{j}).
An argument similar to the one in the proof of Proposition \ref{prop2}, but using assertion (\textit{j}), yields (\textit{jj}).
\end{proof}

\begin{claim}
\label{claim1}
There exists a sequence $\{\hat{v}_n\} \subset \conv\{v_k\}$ such that \eqref{conv_80} holds, and for every $r \geqslant 1$
$$
(\zeta-\Delta)^{-\frac{1}{4}}\eta_r(\hat{v}_n - \sigma) \cdot \nabla (\zeta-\Delta)^{-\frac{3}{4}} \overset{s}{\rightarrow} 0 \text{ in } L^2, \quad \Real \zeta \geqslant \lambda.
$$
{\rm(}here and below we use shorthand $\hat{v}_n - \sigma:=\hat{v}_n\mathcal L^d - \sigma${\rm)}.
\end{claim}

\begin{proof}[Proof of Claim \ref{claim1}]
In view of Claim \ref{jclaim}(\textit{j}), (\textit{jj}), it suffices to establish this convergence over $\mathcal S$. Let $c(x)=e^{-x^2}$, so that $c \in \mathcal S$,  $|(\zeta-\Delta)^{-\frac{1}{4}}c|>0$ on $\mathbb R^d$.

\textit{Step 1.}~Let $r=1$, so $\eta_r=\eta_1$.
Let us show that there exists a sequence $\{v_{\ell_1}^1\} \subset \conv\{v_k\}$ such that
\begin{equation}
\label{zeta_conv}
(\lambda-\Delta)^{-\frac{1}{4}} \eta_1(v^1_{\ell_1} - \sigma) \cdot \nabla (\lambda-\Delta)^{-\frac{3}{4}} \overset{s}{\rightarrow} 0 \text{ in } L^2 \text{ as } \ell_1 \uparrow \infty.
\end{equation}

First, we show that
\begin{equation}
\label{weak_conv}
(\lambda-\Delta)^{-\frac{1}{4}}\eta_1(v_k - \sigma) (\lambda-\Delta)^{-\frac{1}{4}}c \overset{w}{\rightarrow} 0 \text{ in } L^2.
\end{equation}
Indeed, by Claim \ref{jclaim}(\textit{j}), (\textit{jj}),
$\|(\lambda-\Delta)^{-\frac{1}{4}}\eta_1(v_k - \sigma)  (\lambda-\Delta)^{-\frac{1}{4}}c\|_2 \leqslant   2\delta \|c\|_2$ for all $k$. Hence, there exists a subsequence of $\{v_k\}$ (without loss of generality, it is $\{v_k\}$ itself) such that
$
(\lambda-\Delta)^{-\frac{1}{4}}\eta_1(v_k - \sigma) (\lambda-\Delta)^{-\frac{1}{4}}c \overset{w}{\rightarrow} h$ for some $h \in L^2$.
Therefore, given any $f \in \mathcal S$, we have
$
\langle f, (\lambda-\Delta)^{-\frac{1}{4}}\eta_1(v_k - \sigma) (\lambda-\Delta)^{-\frac{1}{4}}c \rangle  \rightarrow \langle f,h\rangle.
$
Along with that, since $v_k \mathcal L^d \overset{w}{\rightarrow} \sigma$, we also have
\begin{align*}
\langle f, (\lambda-\Delta)^{-\frac{1}{4}}\eta_1(v_k - \sigma) (\lambda-\Delta)^{-\frac{1}{4}}c \rangle 
= \langle (\lambda-\Delta)^{-\frac{1}{4}} f, \eta_1(v_k - \sigma) (\lambda-\Delta)^{-\frac{1}{4}}c \rangle \rightarrow 0,
\end{align*}
i.e.~$\langle f,h\rangle=0$. Since $f \in \mathcal S$ was arbitrary, we have $h=0$, which yields \eqref{weak_conv}. 

Now, in view of \eqref{weak_conv}, by Mazur's Theorem, there exists a sequence $\{v_{\ell_1}^1\} \subset \conv\{v_k\}$ such that
\begin{equation}
\label{mazur_conv}
(\lambda-\Delta)^{-\frac{1}{4}}\eta_1(v^1_{\ell_1} - \sigma)  (\lambda-\Delta)^{-\frac{1}{4}}c \overset{s}{\rightarrow} 0 \text{ in } L^2.
\end{equation}
We may assume without loss of generality that each $v^1_{\ell_1} \in \conv \{v_{n}\}_{n \geqslant \ell_1}$. 


Next, set $\ell:=\ell_1$, $\varphi_{\ell}:=\eta_1(v^1_{\ell} - \sigma)$, $\Phi:=(\lambda-\Delta)^{-\frac{1}{4}}c$, fix some $u \in \mathcal S$. We estimate: 
\begin{align*}
&\|(\lambda-\Delta)^{-\frac{1}{4}} \varphi_{\ell} \cdot \nabla (\lambda-\Delta)^{-\frac{3}{4}}u\|_2^2 \\ 
& = 
\left\langle \varphi_{\ell} \cdot \nabla (\lambda-\Delta)^{-\frac{3}{4}}u, (\lambda-\Delta)^{-\frac{1}{2}} \varphi_{\ell} \cdot \nabla (\lambda-\Delta)^{-\frac{3}{4}}u\right\rangle \\
& \text{$\biggl(\text{since }\varphi_\ell \equiv 0$ on $\{|x| \geqslant 2\}$, in the left multiple $\varphi_\ell=\varphi_{\ell} \Phi \frac{\eta_2}{\Phi}\biggr)$} \\
&= \left\langle \varphi_{\ell}  \Phi \frac{\eta_2}{\Phi} \cdot \nabla (\lambda-\Delta)^{-\frac{3}{4}}u, (\lambda-\Delta)^{-\frac{1}{2}} \varphi_{\ell} \cdot \nabla (\lambda-\Delta)^{-\frac{3}{4}}u\right\rangle \\
&=  \left\langle \varphi_{\ell}   \Phi, \frac{\eta_2}{\Phi} \nabla (\lambda-\Delta)^{-\frac{3}{4}}u \left[ (\lambda-\Delta)^{-\frac{1}{2}} \varphi_{\ell} \cdot \nabla (\lambda-\Delta)^{-\frac{3}{4}}u \right]\right\rangle \\
\qquad &\quad(\text{here we are using in the left multiple that } \varphi_\ell=(\lambda-\Delta)^{\frac{1}{4}} (\lambda-\Delta)^{-\frac{1}{4}} \varphi_\ell) \\
&=\biggl\langle (\lambda-\Delta)^{-\frac{1}{4}} \varphi_{\ell} \Phi, (\lambda-\Delta)^{\frac{1}{4}}(fg_\ell) 
\biggr\rangle
\end{align*}
where we set $f:=\frac{\eta_2}{\Phi} \nabla (\lambda-\Delta)^{-\frac{3}{4}}u \in C_0^\infty(\mathbb R^d,\mathbb R^d), 
$
$ g_\ell:=(\lambda-\Delta)^{-\frac{1}{2}} \varphi_{\ell} \cdot \nabla (\lambda-\Delta)^{-\frac{3}{4}}u \in (\lambda-\Delta)^{-\frac{1}{4}}L^2$ (in view of Claim \ref{jclaim}(\textit{j}), (\textit{jj})).
Thus, in view of the above estimates,
\begin{align*}
\|(\lambda-\Delta)^{-\frac{1}{4}} & \varphi_{\ell} \cdot \nabla (\lambda-\Delta)^{-\frac{3}{4}}u\|_2^2 \leqslant \|(\lambda-\Delta)^{-\frac{1}{4}} \varphi_{\ell} \Phi\|_2 \|(\lambda-\Delta)^{\frac{3}{4}}(fg_\ell)\|_2.
\end{align*}
By the Kato-Ponce inequality of \cite[Theorem 1]{GO},
$$
\|(\lambda-\Delta)^{\frac{1}{4}}(fg_\ell)\|_2 \leqslant C\biggl(\|f\|_\infty \|(\lambda-\Delta)^{\frac{1}{4}}g_\ell\|_2 + \|(\lambda-\Delta)^{\frac{1}{4}}f\|_\infty \|g_\ell\|_2\biggr),
$$
for some $C=C(d)<\infty$.
Clearly, $\|f\|_\infty$, $\|(\lambda-\Delta)^{\frac{1}{4}}f\|_\infty <\infty$, and $\|(\lambda-\Delta)^{\frac{1}{4}}g_\ell\|_2$, $\|g_\ell\|_2$ are uniformly (in $\ell$) bounded from above according to Claim \ref{jclaim}(\textit{j}), (\textit{jj}).
Thus, in view of \eqref{mazur_conv}, we obtain \eqref{zeta_conv} (recalling that $\ell_1=\ell$, and $\varphi_{\ell_1}=\eta_1(v^1_{\ell_1} - \sigma)$).

\textit{Step 2.}~Now, we can repeat the argument of Step 1, but starting with sequence $\{v^1_{\ell_1}\}$ in place of $\{v_l\}$, thus obtaining a sequence $\{v^2_{\ell_2}\} \subset \conv \{v^1_{\ell_1}\}$
such that
$$
(\lambda-\Delta)^{-\frac{1}{4}} \eta_2(v^2_{\ell_2} - \sigma) \cdot \nabla (\lambda-\Delta)^{-\frac{3}{4}} \overset{s}{\rightarrow} 0 \text{ in } L^2 \text{ as } \ell_2 \uparrow \infty.
$$
We may assume without loss of generality that each $v^2_{\ell_2} \in \conv \{v_{\ell_1}^1\}_{\ell_1 \geqslant \ell_2}$. Therefore, we also have
$$
(\lambda-\Delta)^{-\frac{1}{4}} \eta_1(v^2_{\ell_2} - \sigma) \cdot \nabla (\lambda-\Delta)^{-\frac{3}{4}} \overset{s}{\rightarrow} 0 \text{ in } L^2 \text{ as } \ell_2 \uparrow \infty.
$$
Repeating this procedure $n-2$ times, we obtain a sequence $\{v^n_{\ell_n}\} \subset \conv\{v^{n-1}_{\ell_{n-1}}\}~(\subset \conv \{v_k\})$ such that
$$
(\lambda-\Delta)^{-\frac{1}{4}} \eta_i(v^n_{\ell_n} - \sigma) \cdot \nabla (\lambda-\Delta)^{-\frac{3}{4}} \overset{s}{\rightarrow} 0 \text{ in } L^2 \text{ as } \ell_n \uparrow \infty, \quad 1 \leqslant i \leqslant n.
$$

\textit{Step 3.}~We set 
$
\hat{v}_n:=v^n_{\ell_n},$ $n \geqslant 1,
$
so for every $r \geqslant 1$
\begin{equation}
\label{zeta_conv2}
(\lambda-\Delta)^{-\frac{1}{4}}\eta_r(\hat{v}_n - \sigma) \cdot \nabla (\lambda-\Delta)^{-\frac{3}{4}} \overset{s}{\rightarrow} 0 \text{ in } L^2.
\end{equation}
Since $v^n_{\ell_n} \in \conv \{v^{n-1}_{\ell_{n-1}}\}_{\ell_{n-1} \geqslant \ell_{n}}$,  $v^{n-1}_{\ell_{n-1}} \in \conv \{v^{n-2}_{\ell_{n-2}}\}_{\ell_{n-2} \geqslant \ell_{n-1}}$, etc, we obtain that
$
v^n_{\ell_n} \in \conv\{v_k\}_{k \geqslant \ell_n},
$
i.e.~we also have \eqref{conv_80}. Finally, \eqref{zeta_conv2} combined with the resolvent identity yield
$$
(\zeta-\Delta)^{-\frac{1}{4}}\eta_r(\hat{v}_n - \sigma) \cdot \nabla (\zeta-\Delta)^{-\frac{3}{4}} \overset{s}{\rightarrow} 0 \text{ in } L^2, \quad \Real\zeta \geqslant \lambda.
$$
i.e.~we have proved Claim \ref{claim1}.
\end{proof}

We are in a position to complete the proof of Proposition \ref{lem5}. Let us show 
that, for every $\zeta \in \mathcal O$
$$
Z_2(\zeta,\hat{v}_n\mathcal L^d)g - Z_2(\zeta,\sigma)g=(\zeta-\Delta)^{-\frac{1}{4}}(\hat{v}_n-\sigma)\cdot\nabla (\zeta-\Delta)^{-\frac{3}{4}}g \overset{s}{\rightarrow} 0 \text{ in }L^2, \quad g \in \mathcal S.
$$
Let us fix some $g \in \mathcal S$. We have
\begin{align*}
(\zeta-\Delta)^{-\frac{1}{4}}(\hat{v}_n-\sigma)\cdot\nabla (\zeta-\Delta)^{-\frac{3}{4}}g 
&= 
(\zeta-\Delta)^{-\frac{1}{4}}(\hat{v}_n  -\eta_r \hat{v}_n)\cdot\nabla (\zeta-\Delta)^{-\frac{3}{4}}g  \\
&+ (\zeta-\Delta)^{-\frac{1}{4}}(\eta_r \hat{v}_n - \eta_r \sigma)\cdot\nabla (\zeta-\Delta)^{-\frac{3}{4}}g \\
&+  (\zeta-\Delta)^{-\frac{1}{4}}(\eta_r \sigma -\sigma)\cdot\nabla (\zeta-\Delta)^{-\frac{3}{4}}g   =:I_{1,r,n}+I_{2,r,n}+I_{3,r}. 
\end{align*}

\begin{claim}
\label{claim_est}
Given any $\varepsilon>0$, there exists $r$ such that
$
\|I_{3,r}\|_2,~~ \|I_{1,r,n}\|_2<\varepsilon, \text{ for all } n
$, $\zeta \in \mathcal O$.
\end{claim}
\begin{proof}[Proof of Claim \ref{claim_est}]
It suffices to prove $\|I_{1,r,n}\|_2<\varepsilon \text{ for all } n.$
We will need the following elementary estimate:
$|\nabla (\zeta-\Delta)^{-\frac{3}{4}}(x,y)| \leqslant M_d (\kappa_{d}^{-1}\Real\zeta-\Delta)^{-\frac{1}{4}}(x,y),$ $x, y \in \mathbb R^d,$ $x \neq y.$
We have
\begin{align*}
\|I_{1,r,n}\|_2&=\|(\Real\zeta-\Delta)^{-\frac{1}{4}} (1-\eta_r) \hat{v}_n \cdot \nabla (\Real\zeta-\Delta)^{-\frac{3}{4}}g\|_2 \\ 
&\leqslant c_d M_d \|(\Real\zeta-\Delta)^{-\frac{1}{4}}(1-\eta_r) |\hat{v}_n| (\kappa_{d}^{-1}\Real\zeta-\Delta)^{-\frac{1}{4}} g\|_2 \\
&\leqslant c_d M_d \bigl\|(\Real\zeta-\Delta)^{-\frac{1}{4}}|\hat{v}_n|^{\frac{1}{2}}\bigr\|_{2 \rightarrow 2} \bigl\|(1-\eta_r) |\hat{v}_n|^{\frac{1}{2}} (\kappa_{d}^{-1}\Real\zeta-\Delta)^{-\frac{1}{4}} g\bigr\|_2
\end{align*}
We have $\bigl\|(\Real\zeta-\Delta)^{-\frac{1}{4}}|\hat{v}_n|^{\frac{1}{2}}\bigr\|_{2 \rightarrow 2} \leqslant \delta$ in view of Lemma \ref{approx_lem}(1). In turn,
\begin{align*}
(1-\eta_r) |\hat{v}_n|^{\frac{1}{2}} & (\kappa_{d}^{-1} \Real\zeta-\Delta)^{-\frac{1}{4}} g \\ &= |\hat{v}_n|^{\frac{1}{2}}(\kappa_{d}^{-1}\Real\zeta-\Delta)^{-\frac{1}{4}}\,(\kappa_{d}^{-1}\Real\zeta-\Delta)^{\frac{1}{4}}(1-\eta_r)(\kappa_{d}^{-1}\Real\zeta-\Delta)^{-\frac{1}{4}}g,
\end{align*}
so
\begin{align*}
\bigl\|(1-\eta_r) |\hat{v}_n|^{\frac{1}{2}} &(\kappa_{d}^{-1} \Real\zeta-\Delta)^{-\frac{1}{4}} g\bigr\|_2 \leqslant \delta \|(\kappa_{d}^{-1} \Real\zeta-\Delta)^{\frac{1}{4}}(1-\eta_r)(\kappa_{d}^{-1} \Real\zeta-\Delta)^{-\frac{1}{4}}g\|_2,
\end{align*}
where $\delta \|(\kappa_{d}^{-1} \Real\zeta-\Delta)^{\frac{1}{4}}(1-\eta_r)(\kappa_{d}^{-1} \Real\zeta-\Delta)^{-\frac{1}{4}}g\|_2 \rightarrow 0 \text{ as } r \uparrow \infty$.
The proof of Claim \ref{claim_est} is completed.
\end{proof}

Claim \ref{claim1}, which 
yields convergence $\|I_{2,r,n}\|_2 \rightarrow 0$ as $n \uparrow \infty$ for every $r$, and Claim \ref{claim_est}, imply that 
$$
Z_2(\zeta,\hat{v}_n\mathcal L^d)g - Z_2(\zeta,\sigma)g \overset{s}{\rightarrow} 0 \text{ in } L^2, \quad g \in \mathcal S, \quad \zeta \in \mathcal O,
$$
which, in view of Claim \ref{jclaim}(\textit{j}), (\textit{jj}),
yields $Z_2(\zeta,\hat{v}_n\mathcal L^d) - Z_2(\zeta,\sigma) \overset{s}{\rightarrow} 0$, $\zeta \in \mathcal O$, in $L^2$ ($\Rightarrow $\eqref{conv_5}). By Claim \ref{claim1}, we also have \eqref{conv_80}. This completes the proof of Proposition \ref{lem5}.
\end{proof}



%

\begin{proposition}
\label{lem50}
Let $p \in \mathcal I$. There exist constants $C_{p}$, $C_{p,q,r}<\infty$ such that
for every $\zeta \in \mathcal O$ 

\smallskip
{\rm(1)} $
\|\Omega_p(\zeta,\sigma,q,r)\|_{p \rightarrow p} \leqslant C_{p,q,r}
$ for all $k$,

\smallskip
{\rm(2)} $
 \|\Omega_p(\zeta,\sigma,\infty,1)\|_{p \rightarrow p} \leqslant C_{p}|\zeta|^{-\frac{1}{2}},
$ for all $k$.

\end{proposition}
\begin{proof} Immediate from Proposition \ref{prop_first} and Proposition \ref{lem5}.
\end{proof}



%
%

Now, we assume that $p \in \mathcal J \subsetneq \mathcal I$.

\begin{proposition}
\label{lem50_p}

Let $\{\hat{v}_n\}$ be the sequence in Proposition \ref{lem5}. For any $p \in \mathcal J$,
$$
\Omega_p(\zeta,\hat{v}_n \mathcal L^d,q,r)  \overset{s}{\rightarrow}  \Omega_p(\zeta,\sigma,q,r) \text{ in }L^p, \quad \zeta \in \mathcal O. 
$$
\end{proposition}
\begin{proof}
Set $\Omega_p \equiv \Omega_p(\zeta,\sigma,q,r)$, $\Omega_p^n \equiv \Omega_p(\zeta,\hat{v}_n\mathcal L^d,q,r)$.  Recall that since $p \in \mathcal J$, we have $2(p-1) \in \mathcal I$. 
Since $\Omega_p$, $\Omega_p^n \in \mathcal B(L^p)$, it suffices to prove convergence on $\mathcal S$. We have ($f \in \mathcal S$):
\begin{equation}
\label{hold_est}
\|\Omega_pf - \Omega_p^nf\|_p^p \leqslant  \|\Omega_pf - \Omega_p^nf\|_{2(p-1)}^{p-1} \|\Omega_pf - \Omega_p^nf\|_2.
\end{equation}
Let us estimate the right-hand side in \eqref{hold_est}:

1) $\Omega_p f - \Omega_p^n f~\bigl(= \Omega_{2(p-1)}f - \Omega_{2(p-1)}^n f\bigr)$ is uniformly bounded in $L^{2(p-1)}$ by Proposition \ref{prop_first} and Proposition \ref{lem50},

2) $\Omega_pf - \Omega_p^n f = \Omega_2f - \Omega_2^n f  \overset{s}{\rightarrow}  0$ in $L^2$ as $k \uparrow \infty$ by Proposition \ref{lem5}.

\smallskip

Therefore, by \eqref{hold_est},
$\Omega_p^n f  \overset{s}{\rightarrow}  \Omega_pf$ in $L^p$, as needed.
\end{proof}

This completes the proof of assertion (\textit{i}), and thus the proof of Theorem \ref{cor0}.

\section{Proof of Theorem \ref{cor1}} 


(\textit{i}) The approximating vector fields $v_k$ were constructed in Section \ref{approxsect}.
The proof repeats the proof of \cite[Theorem 2]{Ki}. 
Namely, we verify conditions of the Trotter approximation theorem for $\Lambda_{C_\infty}(v_k):=-\Delta+v_k \cdot \nabla$, $D(\Lambda_{C_\infty}(v_k))=C^{2} \cap C_\infty$:



%


\smallskip

$1^{\circ}$) $\sup_n\|(\mu+\Lambda_{C_\infty}(v_k))^{-1}\|_{\infty \rightarrow \infty} \leqslant \mu^{-1}$, $\mu \geqslant \kappa_{d} \lambda_{\delta}$.

$2^{\circ}$) $\mu (\mu+\Lambda_{C_\infty}(v_k))^{-1} \rightarrow 1$ in $C_\infty$ as $\mu \uparrow \infty$ uniformly in $n$.

$3^{\circ}$) There exists $\text{\small $s\text{-}C_\infty\text{-}$}\lim_n (\mu+\Lambda_{C_\infty}(v_k))^{-1}$ for some $\mu \geqslant \kappa_{d} \lambda$.

\smallskip

$1^{\circ}$) is immediate.
Let us verify  $2^{\circ}$) and $3^{\circ}$).
Fix some $p \in \mathcal J$, $p>d-1$ (such $p$ exists since $m_{d}\delta<\frac{2d-5}{(d-2)^2}$), and
let 
\begin{equation}
\label{sob_repr2}
\Theta_p(\mu,\sigma):=(\mu-\Delta)^{-\frac{1}{2}-\frac{1}{2q}} \Omega_p(\mu,\sigma,q,1) \in \mathcal B(L^p), \quad \mu \geqslant \kappa_{d}\lambda,
\end{equation}
where $\max\{2,p\}<q$, see the proof of Theorem \ref{cor0}. We will be using the properties of $\Theta_p(\mu,\sigma)$ established there.
Without loss of generality, we may assume that $\{v_k\}$ is the sequence constructed in Proposition \ref{lem50_p}, that is, $v_k \overset{w}{\rightarrow} \sigma,$ and $\Omega_p(\mu,v_k\mathcal L^d,q,1) \overset{s}{\rightarrow} \Omega_p(\mu,\sigma,q,1) \text{  in } L^p$ $\text{as } k\uparrow\infty.$

Given any $\gamma<1-\frac{d-1}{p}$ we can select
$q$ sufficiently close to $p$ so that  by the Sobolev embedding theorem,
$$(\mu-\Delta)^{-\frac{1}{2}-\frac{1}{2q}} [L^p] \subset C^{0,\gamma} \cap L^p, \quad \text{ and }\quad (\mu-\Delta)^{-\frac{1}{2}-\frac{1}{2q}}  \in \mathcal B(L^p,C_\infty).$$
Then Proposition \ref{lem50_p} yields
$\Theta_p(\mu,\hat{v}_n\mathcal L^d)
f  \overset{s}{\rightarrow}   \Theta_p(\mu,\sigma) f$ in $C_\infty$, $f \in \mathcal S,$
as $n \uparrow \infty$. The latter, combined with the next proposition and $1^{\circ}$), verifies condition $3^{\circ}$):

\begin{proposition}
For every $k=1,2,\dots$, $\Theta_p(\mu,v_k\mathcal L^d) \mathcal S \subset \mathcal S$, and
$$
(\mu+\Lambda_{C_\infty}(v_k \mathcal L^d))^{-1}|_{\mathcal S} = \Theta_p(\mu,v_k\mathcal L^d)|_{\mathcal S}, \quad \mu \geqslant \kappa_{d}\lambda.
$$
\end{proposition}
\begin{proof}
The proof repeats the proof of {\cite[Prop.~6]{Ki}}.
\end{proof}

\begin{proposition}
$
\mu\Theta_p(\mu,v_k) \overset{s}{\rightarrow}1 \text{ in } C_\infty \text{ as }\mu \uparrow \infty \text{ uniformly in $k$}.
$
\end{proposition}
\begin{proof}
The proof repeats the proof of {\cite[Prop.~8]{Ki}}.
\end{proof}

The last two proposition yield $2^{\circ}$).
This completes the proof of assertion (\textit{i}).

(\textit{ii}) follows from the equality $\Theta_p(\mu,\sigma)|_{\mathcal S}=(\mu+\Lambda_{C_\infty}(C_\infty))^{-1}|_{\mathcal S}$ (by construction), representation \eqref{sob_repr2}, and the Sobolev embedding theorem.

(\textit{iii})
It follows from (\textit{i}) that $e^{-t\Lambda_{C_\infty}(\sigma)}$ is positivity preserving.
The latter, $1^{\circ}$) and  the Riesz-Markov-Kakutani representation theorem imply (\textit{iii}).


\begin{thebibliography}{99}


\bibitem[AS]{AS} M. Aizenman,  B. Simon. \newblock Brownian motion and Harnack inequality for Schr\"{o}dinger operators. \newblock {\em Comm. Pure. Appl. Math.}, 35 (1982), p.~209-273.


\bibitem[AKR]{AKR} S.~Albeverio, Y.~G.~Kondratiev, M.~Röckner. 
\newblock Stong Feller properties for distorted Brownian motion and applications to finite particle systems with singular interactions. \newblock {\em Contemp.~Math.}, 317 (2003).

\bibitem[AM]{AM} S.~Albeverio, Z.~Ma. \newblock Perturbation of Dirichlet forms --
lower  semiboundedness,  closability,  and form cores. \newblock {\em J.~Funct. Anal.}, 99 (1991), p.~332-356.

\bibitem[BC]{BC} R.~Bass, Z.-Q.~Chen. \newblock Brownian motion with singular drift.
\newblock {\em Ann.~Prob.}, 31 (2003), p.~791-817.






\bibitem[BJ]{BJ}
K.~Bogdan, T.~Jakubowski.
\newblock Estimates of heat kernel of fractional Laplacian
perturbed by gradient operators.
\newblock {\em Comm.~Math.~Phys.}, 271 (2007), p.~179-198.

\bibitem[BL]{BL}
J.~Bourgain, D.~Li.
\newblock On an end-point Kato-Ponce inequality.
\newblock {\em Diff.~Integr.~Equ.}, 27 (2014), p.~1037-1072.

\bibitem[CKS]{CKSo} Z.-Q.~Chen, P.~Kim, R.~Song. 
\newblock Dirichlet heat kernel estimates for fractional Laplacian with gradient perturbation.
\newblock {\em Ann.~Prob.}, 40 (2012), p.~2483-2538.


\bibitem[GO]{GO}
L.~Grafakos,  S.~Oh.
\newblock The Kato-Ponce inequality.
\newblock {\em Comm.~Partial~Diff.~Equ.}, 39 (2014), p.~1128-1157.






\bibitem[KSo]{KSo}
P.~Kim, R.~Song. 
\newblock Stable process with singular drift.
\newblock {\em Stoc.~Proc.~Appl.}, 124 (2014), p.~2479-2516.





\bibitem[Ki]{Ki}
D.~Kinzebulatov.
\newblock A new approach to the $L^p$-theory of $-\Delta + b \cdot \nabla$, and its applications to Feller processes with general drifts.
\newblock {\em Preprint, arXiv:1502.07286} (2015), 18 p.



\bibitem[KPS]{KPS}
V.F.~Kovalenko, M.A.~Perelmuter,  Yu.A.~Semenov.
\newblock Schr\"{o}dinger operators with ${L\sp{1/2}\sb{W}}( R\sp{l})$-potentials.
\newblock {\em J. Math. Phys.}, 22, 1981, p.~1033-1044.


\bibitem[KS]{KS} V.~F.~Kovalenko, Yu.A.~Semenov. 
\newblock $C_0$-semigroups in $L^p(\mathbb R^d)$ and $C_\infty(\mathbb R^d)$ spaces generated by differential expression $\Delta+b\cdot\nabla$. 
(Russian) {\em Teor. Veroyatnost. i Primenen.}, 35 (1990), p.~449-458; English transl. in {\em Theory Probab. Appl.} 35 (1991), p.~443-453.




\bibitem[KR]{KR} N.V.~Krylov and M.~R\"{o}ckner. 
\newblock Strong solutions of stochastic equations with singular time dependent drift. 
\newblock {\em Probab. Theory Related Fields}, 131 (2005), p.~154-196.


\bibitem[S]{S} Yu.A.~Semenov. \newblock On perturbation theory for linear elliptic and parabolic operators; the method of Nash.
\newblock {\em Contemp. Math.}, 221 (1999), p.~217-284.


\bibitem[S2]{S2} Yu.A.~Semenov. \newblock Regularity theorems for parabolic equations. 
\newblock {\em J.~Funct. Anal.}, 231 (2006), p.~375-417.

\bibitem[SV]{SV} P.~Stollmann, J.~Voigt. \newblock Perturbation of Dirichlet forms by measures. 
\newblock {\em Potential Anal.}, 5 (1996), p.~109-138.


\bibitem[Zh]{Zh} Q.~S.~Zhang. \newblock  Gaussian bounds for the fundamental solutions of $\nabla ( A \nabla u ) + B \nabla u - u_t = 0$.
\newblock {\em Manuscripta Math.}, 93 (1997), p.~381-390.




\end{thebibliography}
\end{document}